\documentclass[12pt]{amsart}
\usepackage{fullpage}
\usepackage[american]{babel}

\newcommand{\eps}{\varepsilon}

\newtheorem{lemma}{Lemma}
\newtheorem{proposition}{Proposition}
\newtheorem{theorem}{Theorem}

\begin{document}
\title{Some Hilbert spaces related with the Dirichlet space}

\author{Nicola Arcozzi} \address{Universit\`a di Bologna, Dipartimento di Matematica, Piazza di Porta S.Donato 5, Bologna} \email{nicola.arcozzi@unibo.it}

\author{Pavel Mozolyako} \address{Chebyshev Lab at St. Petersburg State University, 14th Line 29B, Vasilyevsky Island, St.Petersburg 199178, RUSSIA} \email{pmzlcroak@gmail.com} 

\author{Karl-Mikael Perfekt} \address{Department of Mathematical Sciences, Norwegian University of Science and Technology (NTNU), NO-7491 Trondheim, Norway} \email{karl-mikael.perfekt@math.ntnu.no}

\author{Stefan Richter} \address{Department of Mathematics, The University of Tennessee, Knoxville, TN 37996, USA} \email{richter@math.utk.edu}

\author{Giulia Sarfatti} \address{Istituto Nazionale di Alta Matematica ``F. Severi'', Citt\`a Universitaria, Piazzale Aldo Moro 5, 00185 Roma \newline {\tiny and}  \newline \indent Institut de Math\'ematiques de Jussieu,
Universit\'e Pierre et Marie Curie, 4, place Jussieu, F-75252 Paris, France} \email{giulia.sarfatti@imj-prg.fr}

\date{\today}

{\maketitle}

\begin{abstract}
We study the reproducing kernel Hilbert space with kernel $k^d$, where $d$ is a positive integer and $k$ is the reproducing kernel of the analytic Dirichlet space.
\end{abstract}

\section*{Introduction}
Consider the Dirichlet space $\mathcal{D}$ on the unit disc $\{z\in{\mathbb C}:\ |z|<1\}$ of the complex plane.
It can be defined as the Reproducing Kernel Hilbert Space (RKHS) having kernel
$$
k_z(w)=k(w,z)=\frac{1}{\overline{z}w}\log\frac{1}{1-\overline{z}w}=\sum_{n=0}^\infty\frac{(\overline{z}w)^n}{n+1}.
$$
We are interested in the spaces $\mathcal{D}_d$ having kernel $k^d$, with $d \in \mathbb{N}$. ${\mathcal{D}_d}$ can be thought of in terms of function spaces on polydiscs, following ideas of Aronszajn \cite{Aronszajn}. To explain this point of view, note that the tensor $d$-power ${\mathcal{D}}^{\otimes d}$ of the Dirichlet space has reproducing kernel
$k_d(z_1,\cdots,z_d;w_1,\dots,w_d)=\Pi_{j=1}^dk(z_j,w_j)$. Hence, the space of restrictions of functions in ${\mathcal{D}}^{\otimes d}$ to the diagonal $z_1=\dots=z_d$ has the reproducing kernel $k^d$, and therefore coincides with $\mathcal{D}_d$.

We will provide several equivalent norms for the spaces $\mathcal{D}_d$ and their dual spaces in Theorem \ref{norms}. Then we will discuss the properties of these spaces. More precisely, we will investigate:
\begin{itemize}
 \item[-] $\mathcal{D}_d$ and its dual space $HS_d$ in connection with Hankel operators of Hilbert-Schmidt class on the Dirichlet space $\mathcal{D}$; 
 \item[-] the complete Nevanlinna-Pick property for $\mathcal{D}_d$;
 \item[-] the Carleson measures for these spaces.
\end{itemize}
Concerning the first item, the connection with Hilbert-Schmidt Hankel operators served as our original motivation for studying the spaces $\mathcal{D}_d$.

Note that the spaces $\mathcal{D}_d$ live infinitely close to ${\mathcal{D}}$ in the scale of weighted Dirichlet spaces $\tilde{\mathcal{D}}_s$, defined by the norms
$$
\|\varphi\|_{\tilde{{\mathcal{D}}}_s}^2=\int_{-\pi}^{+\pi}\left|\varphi(e^{it})\right|^2\frac{dt}{2\pi}+\int_{|z|<1}\left|\varphi^\prime(z)\right|^2(1-|z|^2)^s\frac{dA(z)}{\pi},\ 0\le s<1,
$$
where $\frac{dA(z)}{\pi}$ is normalized area measure on the unit disc.

\textbf{Notation:} We use multiindex notation. If $n=(n_1,\dots,n_d)$ belongs to ${\mathbb N}^d$, then $|n|=n_1+\dots+n_d$.
We write $A\approx B$ if $A$ and $B$ are 
quantities that depend on a certain family of variables, and there exist independent constants $0<c<C$ such that $c A\le B\le C A$. 

\section*{Equivalent norms for the spaces $\mathcal{D}_d$ and their dual spaces $HS_d$}

\begin{theorem}\label{norms}
 Let $d$ be a positive integer and let 
$$
a_d(k)=\sum_{|n|=k}\frac{1}{(n_1+1)\dots(n_d+1)}.
$$
Then the norm of a function $\varphi(z)=\sum_{k=0}^\infty\widehat{\varphi}(k)z^k$ in $\mathcal{D}_d$ is
\begin{equation}\label{normone}
\|\varphi\|_{\mathcal{D}_d}=\left(\sum_{k=0}^\infty a_d(k)^{-1}\left|\widehat{\varphi}(k)\right|^2\right)^{1/2}\approx[\varphi]_{d},
\end{equation}
where
\begin{equation}\label{normtwo}
[\varphi]_{d}=\left(\sum_{k=0}^\infty \frac{k+1}{\log^{d-1}(k+2)}\left|\widehat{\varphi}(k)\right|^2\right)^{1/2}. 
\end{equation}
An equivalent Hilbert norm $|[\varphi]|_d\approx[\varphi]_d$ for $\varphi$ in terms of the values of $\varphi$ is given by
\begin{equation}\label{normthree}
|[\varphi]|_d= |\varphi(0)|^2+\left(\int_{{\mathbb D}} |\varphi^\prime(z)|^2\frac{1}{\log^{d-1}\left(\frac{1}{1-|z|^2}\right)}\frac{dA(z)}{\pi}\right)^{1/2}.
\end{equation}
Define now the holomorphic space $HS_d$ by the norm:
\begin{equation}\label{HSone}
\|\psi\|_{HS_d}=\left(\sum_{k=0}^\infty(k+1)^2a_d(k)\left|\widehat{\psi}(k)\right|^2\right)^{1/2}. 
\end{equation}
Then, $HS_d\equiv \left(\mathcal{D}_d\right)^*$ is the dual space of ${\mathcal{D}_d}$ under the duality pairing of ${\mathcal D}$. Moreover,
\begin{eqnarray}\label{HStwo}
\|\psi\|_{HS_d}\approx[\psi]_{HS_d}&:=&\left(\sum_{k=0}^\infty(k+1)\log^{d-1}(k+2)\left|\widehat{\psi}(k)\right|^2\right)^{1/2}\approx\crcr
|[\psi]|_{HS_d}&:=&\left(|\psi(0)|^2+\int_{{\mathbb D}} |\psi^\prime(z)|^2\log^{d-1}\left(\frac{1}{1-|z|^2}\right)\frac{dA(z)}{\pi}\right)^{1/2}.
\end{eqnarray}
Furthermore, the norm can be written as
\begin{equation}\label{b}
 \|\psi\|_{HS_d}^2=\sum_{(n_1, \ldots, n_d)} |\langle e_{n_1}\ldots e_{n_d}  , \psi \rangle_{\mathcal D}|^2,
\end{equation}
where $\{e_n\}_{n=0}^\infty$ is the canonical orthonormal basis of $\mathcal{D}$, $e_n(z) = \frac{z^n}{\sqrt{n+1}}$.
\end{theorem}
The remainder of this section is devoted to the proof of Theorem \ref{norms}.
The expression for $\|\varphi\|_{\mathcal{D}_d}$ in (\ref{normone}) follows by expanding $(k_z)^d$ as a power series. 
The equivalence $\|\varphi\|_{\mathcal{D}_d}\approx[\varphi]_{d}$, as well as $\|\varphi\|_{HS_d}\approx[\varphi]_{HS_d}$, are consequences of the following lemma.
We denote by $c,C$ positive constants which are allowed to depend on $d$ only, whose precise value can change from line to line.
 \begin{lemma}\label{lemma1} For each $d\in \mathbb{N}$ there are constants $c,C>0$ such that for all $k\ge 0$ we have $$c a_d(k) \le\frac{\log^{d-1}(k+2)}{k+1} \le C a_d(k).$$ Consequently, if $t \in (0,1)$, then
$$c \left(\frac{1}{t} \log \frac{1}{1-t}\right)^d \le \sum_{k=0}^\infty \frac{\log^{d-1}(k+2)}{k+1} t^k \le C \left(\frac{1}{t} \log \frac{1}{1-t}\right)^d .$$
   \end{lemma}
\begin{proof}[Proof of Lemma \ref{lemma1}]
We will prove the Lemma by induction on $d \in {\mathbb{N}}$. 
It is obvious for $d=1$. Thus let $d\ge 2$ and suppose the lemma is true for $d-1$. Also we observe that  there is a constant $c>0$ such that for all $k\ge 0$ and $0\le n\le k$ we have
$$ c \log^{d-2} (k+2)\le \log^{d-2} (n+2)+\log^{d-2} (k-n+2) \le 2 \log^{d-2} (k+2).$$ Then for $k\ge 0$
\begin{align*}
 a_d(k)&= \sum_{n_1+\dots+n_d=k}\frac{1}{(n_1+1)\dots(n_d+1)}\\
 &= \sum_{n=0}^k\frac{1}{n+1}\sum_{n_2+\dots+n_d=k-n}\frac{1}{(n_2+1)\dots(n_d+1)}\\
  &\approx \sum_{n=0}^k\frac{1}{n+1}\frac{\log^{d-2}(k-n+2)}{k-n+1} \ \ \text{   by the inductive assumption}\\
  &= \frac{1}{2} \sum_{n=0}^{k}\frac{\log^{d-2}(n+2)+\log^{d-2}(k-n+2)}{(n+1)(k-n+1)}\\
  &\approx \log^{d-2}(k+2) \sum_{n=0}^{k} \frac{1}{(n+1)(k-n+1)} \ \ \ \text{  by the earlier observation}\\
  &= \frac{\log^{d-2}(k+2)}{k+2} \sum_{n=0}^{k} \frac{1}{n+1}+ \frac{1}{k-n+1}\\
  &\approx \frac{\log^{d-1}(k+2)}{k+1}. \qedhere
\end{align*} \end{proof}
Next, we prove the equivalence $[\varphi]_{HS_d}\approx|[\varphi]|_{HS_d}$ which appears in (\ref{HStwo}).
\begin{lemma}\label{lemma2} Let $d \in \mathbb{N}$. Then 
  \[ \int_0^1 t^{k } \left(\frac{1}{t} \log \frac{1}{ 1 - t}\right)^{d-1} d t \approx \frac{\log^{d-1} ( k+2)}{k+1}, \ \ \ \ k \ge d.\]
\end{lemma}
Given the Lemma, we expand
\begin{eqnarray*}
|[\psi]|_{HS_d}^2&=&|\widehat{\psi}(0)|^2+\int_{{\mathbb D}} \left|\sum_{k=1}^\infty\widehat{\psi}(k)kz^{k-1}\right|^2\log^{d-1}\frac{1}{1-|z|^2}\frac{dA(z)}{\pi}\crcr
&=&|\widehat{\psi}(0)|^2+\sum_{k=1}^\infty k^2\left|\widehat{\psi}(k)\right|^2\int_0^1\log^{d-1}\frac{1}{1-t}t^{k-1}dt\crcr
&\approx&|\widehat{\psi}(0)|^2+\sum_{k=1}^\infty k^2\left|\widehat{\psi}(k)\right|^2\frac{\log^{d-1}(k+2)}{k+1}\crcr
&\approx&[\psi]_{HS_d}^2,
\end{eqnarray*}
obtaining the desired conclusion.
\begin{proof}[Proof of Lemma \ref{lemma2}]
The case $d=1$ is obvious, leaving us to consider $d\ge 2$. We will also assume that $k \ge 2.$ Then by Lemma \ref{lemma1} we have 
\begin{eqnarray*}
  \int_0^1 t^{k } \left(\frac{1}{t} \log \frac{1}{1-t}\right)^{d-1} d t & \approx & \int_0^1
  t^{k} \sum_{n = 0}^{\infty}{} \frac{\log^{d - 2} ( n+2)}{n+1} t^n d t\\
  & = & \sum_{n = 0}^{\infty}{} \frac{\log^{d - 2} ( n+2)}{(n+1) ( n + k+1)} = S_1 +
  S_2,
\end{eqnarray*}
where 
\begin{align*}
  S_1 &=  \sum_{n = 0}^{k-1}{} \frac{\log^{d - 2} ( n+2)}{(n+1) ( n + k+1)}
   \approx  \frac{1}{k+1} \sum_{n = 0}^{k-1}{} \frac{\log^{d - 2} (
  n+2)}{n+1}
   \approx  \frac{1}{k+1} \int_1^{k+2} \frac{\log^{d - 2} ( t)}{t} d t \\
  &= \frac{1}{d-1}\frac{\log^{d-1} ( k+2)}{k+1}
\end{align*}
and
\begin{align*}
  S_2 &=  \sum_{n = k}^{\infty}{} \frac{\log^{d - 2} ( n+2)}{(n+1) ( n + k+1)}
   \le  \sum_{n = k+1}^{\infty}{} \frac{\log^{d - 2} ( n+1)}{n^2}
   \le  \sum_{j = 1}^{\infty} \sum_{n = k^j }^{k^{j + 1}-1} \frac{}{}
  \frac{\log^{d - 2} ( n+1)}{n^2}\\
  &\le  \sum_{j = 1}^{\infty} (j+1)^{d-2} \log^{d-2} k \sum_{n = k^j }^{k^{j + 1}-1} \frac{1}{n^2} 
  \le  \log^{d - 2} ( k+2) \sum_{j = 1}^{\infty} (j+1)^{d-2} \int_{k^j-1}^{\infty} \frac{1}{x^2}dx \\
  &= \frac{\log^{d - 2} ( k+2)}{k+1}\sum_{j = 1}^{\infty} (j+1)^{d-2} \frac{k+1}{k^j-1}
  \le \frac{\log^{d - 2} ( k+2)}{k+1}\sum_{j = 1}^{\infty} (j+1)^{d-2} \frac{k+1}{(k-1)k^{j-1}} \\
  &\le  \frac{\log^{d - 2} ( k+2)}{k+1}\sum_{j = 1}^{\infty} (j+1)^{d-2} \frac{3}{2^{j-1}} \
   =  o\left(\frac{\log^{d - 1} ( k+2)}{k+1}\right). \qedhere \end{align*}
  \end{proof}
  
  Now, the duality between $\mathcal{D}_d$ and $HS_d$ under the duality pairing given by the inner product of ${\mathcal D}$ is easily seen by considering $[\cdot]_d$
  and $[\cdot]_{HS_d}$. They are weighted $\ell^2$ norms and duality is established by means of the Cauchy-Schwarz inequality. 
  
  Next we will prove that $[\varphi]_d\approx|[\varphi]|_d$. This is equivalent to proving that the dual space of $HS_d$, with respect to the Dirichlet inner product 
  $\langle \cdot\,,\,\cdot\rangle_{\mathcal D}$, is the Hilbert space with the norm $|[\cdot]|_d$.  
  
  Let $d\in \mathbb{N}$ and set, for $0< t<1$, $w_d(t)=\left(\frac{1}{t} \log \frac{1}{ 1 - t}\right)^{d}$ and, for $0<|z|<1$, $W_d(z)=w_d(|z|^2)$ and $W_d(0)=1$. 
  
  \begin{lemma}\label{lemma3} Let $d\in \mathbb{N}$. Then
$$\int_{1-\eps}^1 w_d(t)dt \cdot \int_{1-\eps}^1 \frac{1}{w_d(t)}dt \approx \eps^2 \ \ \text{ as } \eps \to 0.$$
\end{lemma}
\begin{proof}
Write $\tilde{w}(t)=(\log \frac{1}{1-t})^d$, and note that it suffices to establish the lemma for $\tilde{w}$ in place of $w_d$.
Let $\eps >0$. Then $\tilde{w}$ is increasing in $(0,1)$
and $\tilde{w}(1-\eps^{k+1})=(k+1)^d (\log \frac{1}{\eps})^d$, hence
\begin{align*}
\int_{1-\eps}^1 \tilde{w}(t) dt & = \sum_{k=1}^\infty \int_{1-\eps^k}^{1-\eps^{k+1}} \tilde{w}(t)dt\\
& \le \sum_{k=1}^\infty \tilde{w}(1-\eps^{k+1})(\eps^k-\eps^{k+1})\\
&= \sum_{k=1}^\infty (k+1)^d (\log \frac{1}{\eps})^d \eps^k (1-\eps)\\
&\approx \eps (\log \frac{1}{\eps})^d \frac{1}{(1-\eps)^d}\end{align*}

For  $1/\tilde{w}$ we just notice that it is decreasing and hence
 \begin{align*}
\int_{1-\eps}^1 \frac{1}{\tilde{w}(t)} dt & \le  \frac{1}{\tilde{w}(1-\eps)} \eps = \frac{\eps}{(\log \frac{1}{\eps})^d}
 \end{align*}

Thus as $\eps \to 0$ we have
$$\eps^2\le \int_{1-\eps}^1 \tilde{w}(t) dt \int_{1-\eps}^1 \frac{1}{\tilde{w}(t)} dt = O(\eps^2).$$\end{proof}

For $0<h<1$ and $s\in [-\pi,\pi)$ let $S_h(e^{is})$ be the Carleson square at $e^{is}$, i.e.
$$S_h(e^{is})= \{re^{it}: 1-h<r<1, |t-s|<h\}.$$ A positive function $W$ on the unit disc is said to satisfy the Bekoll\'{e}-Bonami condition (B2) if there exists $c>0$ such that
$$\int_{S_h(e^{is})} W dA \cdot \int_{S_h(e^{is})} \frac{1}{W} dA \le c h^4$$ for every Carleson square  $S_h(e^{is})$. If $d\in \mathbb{N}$ and if $W_d(z)$ is defined as before, then
$$\int_{S_h(e^{is})} W_d dA \cdot \int_{S_h(e^{is})} \frac{1}{W_d} dA=h^2 \int_{1-h}^1w_d(t)dt \cdot \int_{1-h}^1 \frac{1}{w_d(t)}dt \approx h^4$$ by Lemma \ref{lemma3},
at least if $0<h<1/2$. Observe that both $W_d$ and $1/W_d$ are positive and integrable in the unit disc, hence it follows that the estimate holds for all $0< h\le 1$. 

Thus $W_d$ satisfies the condition (B2). Furthermore, note that  if $f(z)= \sum_{k=0}^\infty \hat{f}(k)z^k$ is analytic in the open unit disc, then $$\int_{|z|<1}|f(z)|^2w_d(|z|^2)\frac{dA(z)}{\pi}= \sum_{k=0}^\infty w_k |\hat{f}(k)|^2,$$ where $w_k= \int_0^1t^k w_d(t)dt \approx \frac{\log^{d}(k+2)}{k+1}$.

A special case of Theorem 2.1 of Luecking's paper \cite{Luecking} says that if $W$ satisfies the condition (B2) by Bekoll\'{e} and Bonami \cite{BeBo}, 
then one has a duality between the spaces $L^2_a(WdA)$ and $L^2_a(\frac{1}{W}dA)$ with respect to the pairing given by $\int_{|z|<1} f \overline{g}dA$. Thus, we have
\begin{align*}
\int_{|z|<1}|g(z)|^2\frac{1}{W_d(z)}dA 
    &\approx \sup_{f\ne 0} \frac{\left|\int_{|z|<1}g(z) \overline{f(z)}\frac{dA(z)}{\pi}\right|^2}{\int_{|z|<1}|f(z)|^2 W_d(z)dA}\\
    &= \sup_{f\ne 0} \frac{\left|\sum_{k=0}^\infty \frac{\hat{g}(k)}{(k+1)\sqrt{w_k}} \sqrt{w_k}\overline{\hat{f}(k)}\right|^2}{\sum_{k=0}^\infty w_k|\hat{f}(k)|^2}\\
    &= \sum_{k=0}^\infty \frac{1}{(k+1)^2 w_k}|\hat{g}(k)|^2
\end{align*}
This finishes the proof of \eqref{HStwo}. It remains to demonstrate \eqref{b}. We defer its proof to the next section.

By Theorem \ref{norms} we have the following chain of inclusions:
  \[ \ldots \hookrightarrow HS_{d + 1} \hookrightarrow
     HS_d \hookrightarrow \ldots \hookrightarrow HS_2
     \hookrightarrow HS_1 =\mathcal{D}  =
     \mathcal{D}_1 \hookrightarrow \mathcal{D}_2
     \hookrightarrow \ldots \hookrightarrow \mathcal{D}_d
     \hookrightarrow \mathcal{D}_{d + 1} \hookrightarrow
     \ldots \]
  with duality w.r.t. $\mathcal{D}$ linking spaces with the
  same index. It might be interesting to compare this sequence with the sequence of Banach spaces related to the Dirichlet spaces studied in \cite{ARSrelated}. Note that for $d\ge3$ the reproducing kernel of $HS_d$ is continuous up to the boundary. Hence functions in $HS_d$ extend continuously to the closure of the unit disc, for $d \geq 3$.
\section*{Hilbert-Schmidt norms of Hankel-type operators} 
 Let $\{e_n\}$ be the canonical orthonormal basis of $\mathcal D$, $e_n(z) = \frac{z^n}{\sqrt{n+1}}$. Equation \eqref{b} follows from the computation
 \begin{multline*}
 \sum_{k=0}^\infty\sum_{|n|=k}|\langle e_{n_1}... e_{n_d},\psi\rangle|^2 =\sum_{k=0}^\infty\sum_{|n|=k} \frac{1}{(n_1+1)\cdot ... \cdot (n_d+1)} |\langle z^{n_1} ... z^{n_d},\psi\rangle|^2\\
 =\sum_{k=0}^\infty\sum_{|n|=k} \frac{1}{(n_1+1)\cdot ... \cdot (n_d+1)} |\langle z^{k},\psi\rangle|^2
 =\sum_{k=0}^\infty\sum_{|n|=k} \frac{(k+1)^2}{(n_1+1)\cdot ... \cdot (n_d+1)} |\hat{\psi}(k)|^2\\
 =\sum_{k=0}^\infty (k+1) a_d(k) |\hat{\psi}(k)|^2
 \approx\sum_{k=0}^\infty \frac{\log^{d-1}(k+2)}{k+1}|\hat{\psi}(k)|^2.
 \end{multline*}

Polarizing this expression for $\|\cdot\|_{HS_d}$, the inner product of $HS_d$ can be written
\[\langle \psi_1 , \psi_2 \rangle_{HS_d}=\sum_{(n_1, \ldots, n_d)}\langle \psi_1, e_{n_1}\ldots e_{n_d}  \rangle_{\mathcal D}\langle e_{n_1}\ldots e_{n_d}  , \psi_2 \rangle_{\mathcal D}.\]
Hence, for any $\lambda, \zeta \in {\mathbb D}$,
\begin{align*}
\langle k_\lambda, k_\zeta \rangle_{HS_d}&=\sum_{n\in\mathbb N^d}\langle k_\lambda, e_{n_1}\ldots e_{n_d} \rangle_{\mathcal D}
\langle e_{n_1}\ldots e_{n_d}  , k_\zeta \rangle_{\mathcal D}=\sum_{n\in\mathbb N^d}\overline{e_{n_1}(\lambda)\ldots e_{n_d}(\lambda)}e_{n_1}(\zeta)\ldots e_{n_d}(\zeta)\\
&=\left(\sum_{i=0}^{\infty}\overline{e_i(\lambda)}e_i(\zeta)\right)^d=k_\lambda(\zeta)^d=\langle k_\lambda^d, k_\zeta^d\rangle_{\mathcal{D}_d}.
\end{align*}
That is,
\begin{proposition}\label{naomi}
The map $U:k_\lambda \mapsto k_\lambda^d$ extends to a unitary map $HS_d \to {\mathcal{D}_d}$.
\end{proposition}

When $d = 2$, $HS_2$
contains those functions $b$ for which the Hankel operator $H_b : \mathcal{D} \rightarrow \overline{\mathcal{D}}$, defined by $\langle
H_b e_j, \overline{e_k} \rangle_{\overline{\mathcal{D}}} = \langle e_j e_k, b \rangle_{\mathcal{D}}$, belongs to the Hilbert-Schmidt class.

Analogous interpretations can be given for $d\ge3$, but then function spaces on polydiscs are involved. We consider the case $d=3$, which is representative. Consider first the operator 
$T_b:{\mathcal D}\to\overline{\mathcal D}\otimes\overline{\mathcal D}$ defined by
\begin{equation*}
 \left<T_bf,\overline{g}\otimes\overline{h}\right>_{\overline{\mathcal D}\otimes\overline{\mathcal D}}=\left<fgh,b\right>_{\mathcal D}.
\end{equation*}
The formula uniquely defines an operator, whose action is
\begin{align*} T_bf(z,w) &=\langle T_bf, \overline{k}_z\overline{k}_w \rangle_{\overline{\mathcal D}\otimes\overline{\mathcal D}}\\
&=\langle f k_zk_w,b\rangle_{\mathcal D}\\
&= \sum_{n,m,j}\hat{f}(j)\frac{\overline{z}^n}{n+1}\frac{\overline{w}^m}{m+1} \langle \zeta^{n+m+j},b\rangle_{\mathcal D}\\
&= \sum_{n,m,j}\hat{f}(j) \overline{\hat{b}(n+m+j)} \frac{n+m+j+1}{(n+1)(m+1)}\overline{z}^n\overline{w}^m
\end{align*}
Then, the Hilbert-Schmidt norm of $T_b$ is: 
$$
 \sum_{l,m,n}\left|\left\langle T_be_l,e_me_n\right\rangle_{\overline{\mathcal D}\otimes\overline{\mathcal D}}\right|^2= 
 \sum_{l,m,n}\left|\left\langle e_le_me_n,b\right\rangle_{\mathcal D}\right|^2=\|b\|_{HS_3}^2.
$$
Similarly, we can consider 
$U_b:{\mathcal D}\otimes{\mathcal D}\to\overline{\mathcal D}$ defined by
\begin{equation*}
 \left<U_b(f\otimes g),\overline{h}\right>_{\overline{\mathcal D}}=\left<fgh,b\right>_{\mathcal D}.
\end{equation*}
The action of this operator is given by
$$
U_b(f\otimes g)(\overline{z})=\sum_{l,m,n=0}^\infty\widehat{f}(l)\widehat{g}(m)\frac{(l+m+n+1)\overline{\widehat{b}(l+m+n)}}{n+1}\overline{z}^n.
$$
The Hilbert-Schmidt norm of $U_b$ is still $\|b\|_{HS_3}$.

\section*{Carleson measures for the spaces $\mathcal{D}_d$ and $HS_d$}
The (B2) condition allows us to characterize the Carleson measures for the spaces $\mathcal{D}_d$ and $HS_d$. 
Recall that a nonnegative Borel measure $\mu$ on the open unit disc
is Carleson for the Hilbert function space $H$ if the inequality
$$
\int_{|z|<1}|f|^2d\mu\le C(\mu)\|f\|_{H}^2
$$
holds with a constant $C(\mu)$ which is independent of $f$. The characterization \cite{ARS02} shows that, since the (B2) condition holds, then  

\begin{theorem}\label{carleson}
  For $d \in \mathbb{N}$, a measure $\mu \geq 0$ on $\{|z|<1\}$
  is Carleson for $\mathcal{D}_d$ if and only if for $|a|<1$ we have:
  \[ \int_{\tilde{S} ( a)} \log^{d-1}\left(\frac{1}{1-|z|^2}\right) ( 1 - | z |^2) \mu ( S ( z) \cap S (
     a))^2 \frac{d x d y}{( 1 - | z |^2)^2} \le C_1(\mu) \mu ( S ( a)), \]
  where $S ( a)=\{z: 0<1-|z|<1-|a|,\ |\arg(z\overline{a})|<1-|a|\}$ is the Carleson box with vertex $a$ and $\tilde{S} ( a) = \{z: 0<1-|z|<2(1-|a|),\ |\arg(z\overline{a})|<2(1-|a|)\}$ is
  its ``dilation''.
\end{theorem}
The characterization extends to $HS_2$, with the weight $ \log^{-1}\left(\frac{1}{1-|z|^2}\right)$. Since functions in $HS_d$ are continuous for $d\ge3$, all finite measures are 
Carleson measures for these spaces. Once we know the Carleson measures, we can characterize the multipliers for $\mathcal{D}_d$ in a standard way.

\section*{The complete Nevanlinna-Pick property for $\mathcal{D}_d$} Next, we prove that the spaces $\mathcal{D}_d$ have the Complete Nevanlinna-Pick (CNP) Property. 
Much research has been done on  kernels with the CNP property in the past twenty years, following seminal work of Sarason and Agler.
See the monograph \cite{AgMcC} for a comprehensive and very readable introduction to this topic. We give here a definition which is simple to state, although perhaps 
not the most conceptual. An irreducible kernel $k:X\times X\to{\mathbb{C}}$ has the CNP property if there is a positive definite function $F:X\to{\mathbb{D}}$
and a nowhere vanishing function $\delta:X\to\mathbb{C}$ such that:
$$
k(x,y)=\frac{\overline{\delta}(x)\delta(y)}{1-F(x,y)}
$$
whenever $x,y$ lie in $X$. The CNP property is a property of the kernel, not of the Hilbert space itself.
\begin{theorem}\label{cnp}
 There are norms on $\mathcal{D}_d$ such that the CNP property holds.
\end{theorem}
\begin{proof}A kernel $k:\mathbb{D}\times\mathbb{D}\to\mathbb{C}$ of the form $k(w,z)=\sum_{k=0}^\infty a_k(\overline{z}w)^k$ has the 
CNP property if $a_0=1$ and the sequence $\{a_n\}_{n=0}^\infty$ is 
positive and log-convex:
$$
\frac{a_{n-1}}{a_{n}}\le\frac{a_{n}}{a_{n+1}}.
$$
See \cite{AgMcC}, Theorem 7.33 and Lemma 7.38. Consider $\eta(x)=\alpha\log\log(x)-\log(x)$, with real $\alpha$. Then, $\eta^{\prime\prime}(x)=
\frac{\log^2(x)-\alpha\log(x)-\alpha}{x^2\log^2(x)}$, which is positive for $x\ge M_\alpha$, depending on $\alpha$. Let now
\begin{eqnarray}\label{notte}
 a_n&=&\frac{\log^{d-1}(M_d(n+1))}{\log(M_d)\cdot(n+1)}\approx\frac{1}{n+1}+\frac{\log^{d-1}(n+1)}{n+1}  
\end{eqnarray}
Then, the sequence $\{a_n\}_{n=0}^\infty$ provides the coefficients for a  kernel with the CNP property for the space $\mathcal{D}_d$.
\end{proof}

The CNP property has a number of consequences. For instance, we have that the space $\mathcal{D}_d$ and its multiplier algebra $M(\mathcal{D}_d)$ 
have the same interpolating sequences. Recall that a sequence $Z=\{z_n\}_{n=0}^\infty$ is \it interpolating \rm for a RKHS $H$ with reproducing kernel $k^H$ if the 
weighted restriction 
map $R:\varphi\mapsto\left\{\frac{\varphi(z_n)}{k^H(z_n,z_n)^{1/2}}\right\}_{n=0}^\infty$ maps $H$ boundedly onto $\ell^2$; while $Z$ is interpolating for the multiplier algebra $M(H)$
if $Q:\psi\mapsto\left\{{\psi(z_n)}\right\}_{n=0}^\infty$ maps $M(H)$ boundedly onto $\ell^\infty$. The reader is referred to \cite{AgMcC} and to the second chapter of \cite{SeipMonograph}
for more on this topic.

It is a reasonable guess that the \it universal interpolating sequences \rm for $\mathcal{D}_d$ and for its multiplier space $M(\mathcal{D}_d)$ 
are characterized by a Carleson condition and a separation condition, as described in \cite{SeipMonograph} (see the Conjecture at p. 33). See also \cite{Boe}, which contains the best
known result on interpolation in general RKHS spaces with the CNP property.
Unfortunately we do not have an answer for the spaces $\mathcal D_d$.


\begin{thebibliography}{99}
\bibitem[AMcC]{AgMcC} Agler, Jim; McCarthy, John E.
Pick interpolation and Hilbert function spaces. Graduate Studies in Mathematics, 44. American Mathematical Society, Providence, RI, 2002. xx+308 pp. ISBN: 0-8218-2898-3
 \bibitem[ARS1]{ARS02} Arcozzi, Nicola; Rochberg, Richard; Sawyer, Eric Carleson measures for analytic Besov spaces. Rev. Mat. Iberoamericana 18 (2002), no. 2, 443-510.
  \bibitem[ARS2]{ARSrelated} Arcozzi, N.; Rochberg, R.; Sawyer, E.; Wick, B. D.
Function spaces related to the Dirichlet space.
J. Lond. Math. Soc. (2) 83 (2011), no. 1, 1-18.
 \bibitem[BB]{BeBo} Bekoll\'{e}, David; Bonami, Aline In\'{e}galit\'{e}s \`{a} poids pour le noyau de Bergman. (French) C. R. Acad. Sci. Paris S\'{e}r. A-B 286 (1978),
 no. 18, A775-A778.
 \bibitem[B]{Boe} Boe, Bjarte An interpolation theorem for Hilbert spaces with Nevanlinna-Pick kernel. Proc. Amer. Math. Soc. 133 (2005), no. 7, 2077-2081
 \bibitem[Aro]{Aronszajn} Aronszajn, N. Theory of reproducing kernels. Trans. Amer. Math. Soc. 68, (1950). 337-404.
  \bibitem[L]{Luecking}  Luecking, Daniel H. Representation and duality in weighted spaces of analytic functions. Indiana Univ. Math. J. 34 (1985), no. 2, 319-336.
Hankel operators and invariant subspaces of the Dirichlet space.
J. Lond. Math. Soc. (2) 91 (2015), no. 2, 423-438.
 \bibitem[S]{SeipMonograph} Seip, Kristian Interpolation and sampling in spaces of analytic functions.
 University Lecture Series, 33. American Mathematical Society, Providence, RI, 2004. xii+139 pp. ISBN: 0-8218-3554-8
\end{thebibliography}

\end{document}